\documentclass[11pt, english, a4paper]{amsart}

\usepackage{graphicx}

\graphicspath{{img/}}

\usepackage{amsmath}
\usepackage{amssymb}
\usepackage[utf8]{inputenc}
\usepackage[colorlinks=true, citecolor=blue, linkcolor=blue, urlcolor=blue]{hyperref}
\usepackage{bm} 
\usepackage[T1]{fontenc}
\usepackage{amsfonts}
\usepackage{mathrsfs}
\usepackage{yfonts}
\usepackage{url}
\usepackage{mathpazo}  

\usepackage{mathtools}

\usepackage{graphicx}
\usepackage{verbatim}
\usepackage{enumerate}	
\usepackage[english]{babel}
\usepackage{ebproof}		
\usepackage{quoting}
\usepackage{array,booktabs}
 \usepackage{pifont}

\usepackage[shortlabels]{enumitem}

\usepackage{todonotes}

\usepackage[pdf]{graphviz}

\usepackage[margin=2.5cm]{geometry}

\theoremstyle{definition}
\newtheorem{definition}{Definition}
\newtheorem{example}[definition]{Example}

\newtheorem{remark}[definition]{Remark}

\theoremstyle{plain}
\newtheorem{lemma}[definition]{Lemma}
\newtheorem{proposition}[definition]{Proposition}
\newtheorem{theorem}[definition]{Theorem}
\newtheorem{corollary}[definition]{Corollary}

\makeatletter
\newcommand{\tpmod}[1]{{\@displayfalse\pmod{#1}}}
\makeatother

\newcommand\Ni[1][S]{\mathfrak{I}_{0}(#1)}

\newcommand{\Ap}{\operatorname{Ap}}
\newcommand{\Cl}{\mathcal{C}\!\ell}

\newcommand{\Hm}{H_m}
\newcommand{\Hsm}{H_{m+1}}
\newcommand{\IOm}{\mathfrak{I}_{0}(\Hm)}
\newcommand{\IOsm}{\mathfrak{I}_{0}(\Hsm)}

\usepackage{hyperref}
\usepackage{amsaddr}

\makeatletter
\def\@bignumber#1#2{%
  \ifx#2\end
    #1\let\next\@gobble
  \else
    #1\hspace{0pt plus 1pt}\let\next\@bignumber
  \fi
  \next#2}
\newcommand{\bignumber}[1]{\@bignumber#1\end}
\makeatother

\title[Lattice ideal class monoid]{When the poset of the ideal class monoid of a numerical semigroup is a lattice}

\author{S. Bonzio}
\address{Departament of Mathematics and Computer Science, University of Cagliari, via Ospedale 72, 09124 Cagliari, Italy}
\email{stefano.bonzio@unica.it}

\author{P. A. García-Sánchez}
\address{Departamento de \'Algebra and IMAG, Universidad de Granada, E-18071 Granada, Espa\~na}
\email{pedro@ugr.es}

\begin{document}

\begin{abstract}
    We characterize numerical semigroups for which the poset of its ideal class monoid is a lattice, and study the irreducible elements of such a lattice with respect to union, intersection, infimum and supremum.
\end{abstract}

\maketitle

\section{Introduction}

A \emph{numerical semigroup} is a co-finite submonoid of the monoid of non-negative integers under addition, denoted by $\mathbb{N}$ in this manuscript. The co-finite condition is equivalent to saying that the greatest common divisor of the elements of the semigroup is one. 

A (relative) \emph{ideal} of $S$ is a subset $I$ of integers such that $I+S=I$ and $z+I\subseteq S$ for some integer $z$. This last condition is equivalent to the existence of $\min(I)$. An ideal $I$ is said to be \emph{normalized} if $\min(I)=0$. The set of normalized ideals of $S$ is denoted by $\Ni$.

If $I$ and $J$ are ideals of a numerical semigroup $S$, then $I+J=\{i+j : i\in I, j\in J\}$ is also an ideal of $S$ (see \cite[Chapter~3]{ns-app} for the basic properties of ideals of numerical semigroups). The set of ideals of $S$ under this operation is a monoid, and its identity element is $S$. Notice that $(\Ni,+)$ is a submonoid of this monoid.

Observe that if $I$ is an ideal of $S$ and $z\in \mathbb{Z}$, then $z+I=(z+S)+I$ is also an ideal of $S$. For $I$ and $J$ ideals of $S$, we write $I\sim J$ if $I=z+J$ for some $z\in \mathbb{Z}$. We denote by $\Cl(S)$ the quotient of the set of ideals of $S$ by the binary relation $\sim$, and let $[I]$ be the equivalence class of $I$ under $\sim$. We may think of $\Cl(S)$ as the set of ideals modulo principal ideals. 

Given $I$ and $J$ ideals of $S$, it is easy to see that $[I]+[J]=[I+J]$ is well defined, and so $(\Cl(S),+)$ is a monoid, which is isomorphic to $(\Ni,+)$, via the isomorphism $[I]\mapsto -\min(I)+I$ (see for instance \cite{apery-icm}).
If $S$ and $T$ are numerical semigroups such that $(\Cl(S),+)$ and $(\Cl(T),+)$ are isomorphic, then $S$ and $T$ must be equal \cite{isom-icm}. The set $\Cl(S)$ was introduced in \cite{b-k}, where the firsts bounds on its cardinality were given together with some properties related to reduction number of its elements and the generators of $(\Cl(S),+)$.

If $I$ and $J$ are normalized ideals of a numerical semigroup $S$, and there exists $K\in \Ni$ such that $I+K=J$, then we write $I\preceq J$. The binary relation $\preceq$ is an order relation on $\Ni$, and thus $(\Ni,\preceq)$ is a poset. In the same way, we write $[I]\preceq [J]$ if there exists an ideal $K$ of $S$ such that $[I]+[K]=[J]$. In this way, $(\Cl(S),\preceq)$ becomes a poset isomorphic to $(\Ni,\preceq)$. In \cite{m-3-poni}, the poset $(\Ni,\preceq)$ is studied for $S$ a numerical semigroup of embedding dimension three, and it is shown that it is always a lattice. 

One of the goals of this manuscript is to characterize those numerical semigroups $S$ for which $(\Ni,\preceq)$ is a lattice (Section~\ref{sec:lattice}); the other is the study of irreducible elements with respect to the different operations defined on $\Ni$, when $(\Ni,\preceq)$ is a lattice (Section~\ref{sec:irred}).

The reader will notice that some computations with specific numerical semigroups and ideals are omitted. These computations can be found in the file \emph{lattice-normalized-ideals-ns.ipynb}
on the repository

\centerline{\emph{https://github.com/numerical-semigroups/ideal-class-monoid}}

\noindent and were performed with the help of the \texttt{GAP} \cite{gap} package \texttt{numericalsgps} \cite{numericalsgps}.

\section{Preliminaries}

\subsection{Numerical Semigroups}
Let $S$ be a numerical semigroup. The set $\mathbb{N}\setminus S$ is known as the set of gaps of $S$, denote here by $\operatorname{G}(S)$, and its cardinality is the genus of $S$. 
The largest integer not belonging to $S$ is known as the \emph{Frobenius number} of $S$, which we denote by $\operatorname{F}(S)$. 
The \emph{conductor} of $S$ is the least integer $c$ such that $c+\mathbb{N}\subseteq S$. We will use $\operatorname{C}(S)$ to refer to the conductor of $S$; it follows easily that $\operatorname{C}(S)=\operatorname{F}(S)+1$. The least positive integer belonging to $S$ is known as the \emph{multiplicity} of $S$, denoted by $\operatorname{m}(S)$.

Given $A\subseteq \mathbb{N}$ the smallest submonoid of $\mathbb{N}$ that contains $A$ is 
\[
\langle A\rangle = \left\{
\sum\nolimits_{i=1}^n a_i : n\in \mathbb{N}, a_1,\dots,a_n\in A\right\}.
\]
Clearly, if $S$ is a numerical semigroup, then $\langle S\rangle =S$. If $A\subseteq S$ is such that $\langle A\rangle =S$, then we say that $A$ is a set of generators of $S$, or simply that $A$ generates $S$. We say that $A$ is a minimal set of generators of $S$ if no proper subset of $A$ generates $S$. It is well known that $S$ admits a unique minimal set of generators $S^*\setminus (S^*+S^*)$, with $S^*=S\setminus \{0\}$ (see for instance \cite[Corollary~2]{ns-app}), moreover this set cannot have two elements congruent modulo $\operatorname{m}(S)$, and so the cardinality of the minimal set of generators, known as the embedding dimension of $S$, is always smaller than the multiplicity of $S$. The elements of $S^*\setminus (S^*+S^*)$ are known as minimal generators of $S$. If follows easily that $s\in S$ is a minimal generator if and only if $S\setminus \{s\}$ is a numerical semigroup.

Associated to a numerical semigroup $S$ one can define the order induced (on $\mathbb{Z}$) by $S$ as $a\le_S b$ if $b-a\in S$, for any $a,b\in\mathbb{Z}$. 

If $X$ is a finite set of integers, then $X+S$ is an ideal of $S$. If $X=\{x\}$ for some $x\in \mathbb{Z}$, then we write $x+S$ instead of $\{x\}+S$. Notice that every ideal $I$ of $S$ is of this form, since it suffices to take $X=\operatorname{Minimals}_{\le_S}(I)$, which is know as a \emph{minimal generating set} of $I$, and its elements are called \emph{minimal generators} of $I$. There cannot be two different minimal generators congruent modulo the multiplicity of $S$. In particular, the minimal generating set of an ideal $I$ of $S$ has at most cardinality $\operatorname{m}(S)$.

The union of two ideals of a numerical semigroup is again an ideal, and the same holds for the intersection (see \cite[Chapter~3]{ns-app} for an introduction of the basic properties of ideals of numerical semigroups). 

\subsection{Posets and lattices}\label{sec:posets}

A poset $(P,\leq)$ is a set equipped with a (partial) order relation ($\leq$). Given two elements $a,b$ in a poset $(P,\leq)$, $b$ is a \emph{cover} of $a$ if $a < b$ (that is $a\leq b$ and $a\neq b$) and for every $c$ such that $a\leq c\leq b$, then either $c = a$ or $c = b$. 

A poset $(P,\leq)$ is a lattice if, for every pair of elements $x,y\in P$, there exist the infimum and the supremum of $\{x,y\}$: in such a case they are (binary) operations on the set $P$ usually denoted by $x\wedge y$ and $x\vee y$, and referred to as \emph{meet} and \emph{join}, respectively. A lattice can be equivalently defined as a set equipped with two idempotent, associative, commutative and absorbptive operations $\wedge $ and $\vee$. 

A \emph{meet semilattice} $(S,\wedge)$ is a set $S$ with an associative, commutative and idempotent operation $\wedge$. On every meet semilattice one can define the partial order relation $x\leq y$ if  $x\wedge y = x$. In the case the order $\leq$ has a maximum element 1, then $(S,\wedge)$ is called a meet semilattice with one (1 is, equivalently, the neutral element for the operation $\wedge$). Similarly $(S,\vee)$ is a \emph{join semilattice} if $\vee $ is an associative, commutative and idempotent operation on $S$; in this case, we can define the partial order relation $x\le y$ if $x\vee y = y$. If the induced order $\leq$ has a minimum element $0$, then $(S,\vee)$ is a join semilattice with zero ($0$ is the neutral element for $\vee$). In a join (respectively meet) semilattice the element $x\vee y$ is the supremum (respectively infimum) of the set $\{x,y\}$ with respect to the induced order $\leq$.  

Given a poset $(P,\leq)$ and $X\subseteq P$, we indicate by $\uparrow\!X$ and $\downarrow\!X$ the set of upper and lower bounds, respectively, of $X$, namely $\uparrow\!X =\{a\in P  :  x\leq a, \text{ for every } x\in X\}$ and $\downarrow\!X =\{a\in P  :  a\leq x, \text{ for every } x\in X\}$. For $X = \{x\}$, we will write $\uparrow\!x$ and $\downarrow\!x$ (instead of $\uparrow\!\{x\}$ and $\downarrow\!\{x\}$, respectively). 

Every finite meet semilattice (with one) or join semilattice (with zero) is indeed a lattice by \cite[Theorem 2.4]{Nation}. 

Let $(P,\le)$ be a lattice, and let $p\in P$. We say that $p$ is $\wedge$-\emph{irreducible} if it cannot be expressed as $p=q\wedge r$ with $q,r\in P\setminus\{p\}$. Analogously, we say that $p$ is $\vee$-\emph{irreducible} if there is no $q,r\in P\setminus\{p\}$ such that $p=q\vee r$. 

The next result characterizes irreducibles in terms of covers.

\begin{proposition}\label{prop:irreducible-unique-cover}
    Let $(P,\le)$ be a finite lattice, and let $p\in P$. Then, $p$ is $\wedge$-irreducible if and only if $p$ has at most one cover. Dually, $p$ is $\vee$-irreducible if and only if $p$ covers at most one element in $P$.
\end{proposition}
\begin{proof}
    Suppose that $p$ is $\wedge$-irreducible and that $q$ and $r$ are covers of $p$, with $q\neq r$. Then, $p\le q\wedge r\le q$. As $q$ covers $p$, either $p=q\wedge r$ or $q\wedge r=q$. The first case is not possible, since $p$ is irreducible, while the second leads to $q=q\wedge r\le r$, which cannot be possible since then $p\le q\le r$, and then either $q=r$ or $q=p$ ($r$ is a cover of $p$). 

    For the other implication, if $p$ has no covers, then, since $P$ is finite, $\uparrow\!p=\{p\}$, and this forces $p$ to be $\wedge$-irreducible. Suppose that $q$ is the unique cover of $p$. If $p=r\wedge s$ with $r,s\in P\setminus\{p\}$, then $p\le r$ and $p\le s$, which yields $q\le r$ and $q\le s$, since $q$ is the unique cover of $p$. In particular, $q\le r\wedge s$, leading to $r\wedge s= p< q\le r\wedge s$, a contradiction. 
\end{proof}

\section{Normalized ideals of numerical semigroups}\label{sec: normalized ideals}

Recall that a relative ideal $I$ of a numerical semigroup $S$ is a \emph{normalized} ideal $I$ if $\min(I) = 0$. The set $\mathcal{I}_{0}(S)$ of normalized ideals of $S$ is always finite and forms a (complete) lattice under the operations of $\cap$ and $\cup$ (see \cite{apery-icm}). Moreover, it can be turned into a poset upon considering the following partial order relation: 

\centerline{$I\preceq J$ if there exists $L\in\mathcal{I}_{0}(S)$ such that $I + L = J$.}

\vspace{0.2cm}
We will write $I\prec J$ if $I\preceq J$ and $I\neq J$ (see the paragraph preceding Remark~5.1 in \cite{apery-icm}).

\begin{remark}\label{rem:subset-contained-prec}
  Let $S$ be a numerical semigroup, and let $I,J\in \Ni$. If $I\preceq J$, then $J=I+K$ for some $K\in \Ni$. Thus, $I=I+0\subseteq I+K=J$, and consequently $I\subseteq J$. 
\end{remark}

The map $\Cl(S)\to \Ni$, $[I]\mapsto -\min(I)+I$ is a monoid isomorphism which preserves the order $\preceq$, and consequently the posets $(\Cl(S),\preceq)$ and $(\Ni,\preceq)$ are isomorphic. 

Let $m$ be the multiplicity of $S$. If $I$ is an ideal of $S$, and $x\in I$, then $x+ks\in I$ for every non-negative integer $k$. It follows that the set $\Ap(I)=\{ x\in I : x-m\not \in I\}$ generates $I$ as an ideal, and it has precisely $m$ elements, one per each congruence class modulo $m$. The set $\Ap(I)$ is known as the \emph{Apéry set} of $I$ (with respect to $m$). Thus, $\Ap(I)=\{w_0,w_1,\dots,w_{m-1}\}$, where $w_i=\min(I\cap (i+m\mathbb{N}))$.

As $I\subseteq \mathbb{N}$, we deduce that $\Ap(I)\subseteq \mathbb{N}$, and consequently for every $i\in\{0,\dots,m-1\}$, $w_i=m x_i +i$ for some non-negative integer $x_i$; $i=w_i \bmod m$, the remainder of the division of $w_i$ by $m$. The tuple $(x_1,\dots,x_{m-1})$ is known as the \emph{Kunz coordinates} of $I$ (see \cite[Section~4]{apery-icm}). Notice that we are omitting $x_0$, which is equal to $0$, as $0=\min(I)$. In the sequel we will write $I=(x_1,\dots,x_{m-1})_{\mathcal{K}}$ to denote that $(x_1,\dots,x_{m-1})$ are the Kunz coordinates of $I$.

It is well known (see \cite[Theorem~4.4]{apery-icm}) that $(x_1,\dots,x_{m-1})\in\mathbb{N}^{m-1}$ are the Kunz coordinates of an ideal $I$ of a numerical semigroup $S=(k_1,\dots,k_{m-1})_{\mathcal{K}}$ if and only if $(x_1,\dots,x_{m-1})$ fulfills the following inequalities
\begin{equation}
\label{eq:kunz-inequalities-general}
\begin{cases}
x_i\le k_i, \text{for all } i\in \{1,\dots,m-1\},\\
x_i+k_j+\lfloor \frac{i+j}{m}\rfloor \ge x_{(i+j)\bmod m}, \text{for all } i,j\in \{1,\dots,m-1\},
\end{cases}
\end{equation}
where $\lfloor q\rfloor=\max\{ z\in \mathbb{Z} : z\le q\}$ for every $q\in \mathbb{Q}$. 

Let $n\in \mathbb{N}$, and let $i= n\bmod m$.
Then, $n=km+i$ for some $k\in\mathbb{N}$. We know that $w_i$ is the minimum element in $I$ congruent with $i$ modulo $m$. Hence, $n\in I$ if and only if $n\ge w_i$, or equivalently $k\ge x_i$. In particular, $n\not\in I$ if and only if $k\in \{0,\dots,x_i-1\}$, and this holds for every congruence class modulo $m$. Thus, the number of non-negative integers not belonging to $I$ is
\begin{equation}\label{eq:genus-ideal-kunz}
    |\mathbb{N}\setminus (x_1,\dots,x_{m-1})_{\mathcal{K}}| = x_1+\dots+x_{m-1}.
\end{equation}

With this in mind it is easy to show that if $J$ is an ideal with Kunz coordinates $(y_1,\dots,y_{m-1})_{\mathcal{K}}$, then
\begin{equation}\label{eq:inclusion-kunz-coord}
    I\subseteq J \text{ if and only if } (y_1,\dots,y_{m-1})\le (x_1,\ldots,x_{m-1})
\end{equation}
with respect to the usual partial order on $\mathbb{N}^{m-1}$, and
\begin{align*}
    (x_1,\dots,x_{m-1})_{\mathcal{K}}\cap (y_1,\dots,y_{m-1})_{\mathcal{K}} & =(\max(\{x_1,y_1\}),\dots,\max(\{x_{m-1},y_{m-1}\}))_{\mathcal{K}},\\
    (x_1,\dots,x_{m-1})_{\mathcal{K}}\cup (y_1,\dots,y_{m-1})_{\mathcal{K}} & =(\min(\{x_1,y_1\}),\dots,\min(\{x_{m-1},y_{m-1}\}))_{\mathcal{K}}.
\end{align*}
Addition requires more effort, but can be derived by translating \cite[Proposition~4.8]{apery-icm} to Kunz coordinates. If $I+J=(z_1,\dots,z_{m-1})_{\mathcal{K}}$, then for every $i\in \{1,\dots,m-1\}$
\begin{equation}\label{eq:kunz-coordinates-sum-general}
    z_i=\min(\{x_{i_1}+y_{i_2}+\lfloor \tfrac{i_1+i_2}m \rfloor : i_1,i_2\in \{0,\dots,m-1\}, i_1+i_2\equiv i \tpmod{m}\}).
\end{equation}

Let $I\in \Ni$. Then, $x\in I\setminus\{0\}$ is a minimal generator of $I$ if and only if $I\setminus\{x\}\in \Ni$ (see \cite[Lemma~9]{isom-icm}).

The following result relates the poset of normalized ideals of a numerical semigroup $S$ with that of $S\setminus\{a\}$ with $a$ a minimal generator of $S$ larger than the Frobenius number and the multiplicity of $S$ (illustrated in Figure~\ref{fig:evolution}). 

\begin{proposition}\label{prop:normalized-ideals-unitary-extension}
  Let $S$ be a numerical semigroup with multiplicity $m$ 
  and Kunz coordinates $(k_1,\dots,k_{m-1})$. 
  Let $a$ be a minimal generator of $S$, and set $i=a\bmod m$. If $a>\max\{m,\operatorname{F}(S)\}$, then 
  \begin{itemize}
      \item $\Ni\subsetneq \Ni[S\setminus\{a\}]$; 
      \item   $\Ni[S\setminus\{a\}]\setminus \Ni$ is the set of ideals of $\Ni[S\setminus\{a\}]$ not containing $a$, which is the set of ideals of $\Ni[S\setminus\{a\}]$ whose $i$th Kunz coordinate is $k_i+1$.
  \end{itemize}
 Moreover, for every $I=(x_1,\dots,x_{i-1},k_i+1,x_{i+1})_{\mathcal{K}}\in \Ni[S\setminus\{a\}]$, $I+S=I\cup\{a\}$, that is,
  \[(x_1,\dots,x_{i-1},k_{i}+1,x_{i+1},\dots,x_{m-1})_{\mathcal{K}}+S= (x_1,\dots,x_{i-1},k_i,x_{i+1},\dots,x_{m-1})_{\mathcal{K}}.\]
  In particular, $(x_1,\dots,x_{i-1},k_i,x_{i+1},\dots,x_{m-1})_{\mathcal{K}}$ covers $(x_1,\dots,x_{i-1},k_{i}+1,x_{i+1},\dots,x_{m-1})_{\mathcal{K}}$.
\end{proposition}
\begin{proof}
  First, observe that the multiplicity of $S$ and $S'=S\setminus\{a\}$ is the same, and if $S'=(k_1',\dots,k_{m-1}')_{\mathcal{K}}$, then $k_j=k_j'$ for $i\neq j$, and $k_i'=k_i+1$.

  Notice that if $I\in \Ni$, then $I=I+0\subseteq I+S\setminus\{a\}\subseteq I+S=I$, and so $I\in \Ni[S\setminus\{a\}]$. Also, the $i$th Kunz coordinate of $S'$ is $k_i+1$, which by \eqref{eq:kunz-inequalities-general} means that $S'\not\in \Ni$. This proves that $\Ni\subsetneq \Ni[S\setminus\{a\}]$. 

  Next, we prove that $\Ni[S\setminus\{a\}]\setminus \Ni$ is the set of ideals of $\Ni[S\setminus\{a\}]$ not containing $a$. 
  If $I=(x_1,\dots,x_{m-1})_{\mathcal{K}}$ is an ideal of $\Ni[S\setminus\{a\}]$ and 
  $a\not\in I$, then $x_i>k_i$ and by \eqref{eq:kunz-inequalities-general}, $I$ is not an ideal of $\Ni$. For the other inclusion, let $I\in \Ni[S\setminus\{a\}]\setminus\Ni$. If $a\in I$, then $S\subseteq I$ and $a+i\in S\subseteq I$ for all $i\in I$ (recall that $a>\operatorname{F}(S)$), which proves that $a+I\subseteq I$, and consequently $I+S=I+((S\setminus\{a\}\cup\{a\}) \subseteq I$, yielding $I\in \Ni$, a contradiction. Thus, $a\not\in I$.

  Now, we show that the ideals $I=(x_1,\dots,x_{m-1})_{\mathcal{K}}$ of $\Ni[S\setminus\{a\}]$ not containing $a$ are precisely those with $x_i=k_i+1$. If $I$ does not contain $a$, then $x_i > k_i$ (because $a=k_im+i\not\in I$), and as $I$ is an ideal of $S\setminus\{a\}$, by \eqref{eq:kunz-inequalities-general}, we have $x_i\le k'_i=k_i+1$. This forces $x_i=k_i+1$. The converse also holds: if $x_i=k_i+1$, then $I$ does not contain $a$.

  The semigroup $S$ is an ideal of $S\setminus\{a\}$ (it is an oversemigroup of $S\setminus\{a\}$, and thus it is an idempotent ideal of $S\setminus\{a\}$ by \cite[Proposition~5.14]{apery-icm}). 
  Let $I=(x_1,\dots,x_{m-1})_{\mathcal{K}}\in \Ni[S\setminus\{a\}]$ with its $i$th coordinate equal to $k_{i}+1$. By \eqref{eq:kunz-coordinates-sum-general}, for $j\neq i$, the $j$th coordinate of $I+S$ is the minimum of $\{x_j,k'_j\}\cup \{ x_l+k'_n : l+n\equiv j \pmod{m}, l+n<m \}\cup \{ x_j+k'_n+1 : l+n\equiv j \pmod{m}, l+n>m \}$. Notice that by \eqref{eq:kunz-inequalities-general}, this minimum is precisely $x_j$. The $i$th coordinate of $I+S$ is the minimum of $\{k_i+1,k_i\}\cup \{ x_l+k'_n : l+n\equiv i \pmod{m}, l+n<m \}\cup \{ x_l+k'_n+1 : l+n\equiv i \pmod{m}, l+n>m \}$. This minimum is $k_i$, and so $I+S=(x_1,\dots,x_{i-1},k_i,x_{i+1},\dots,x_{m-1})_{\mathcal{K}}$.
 It follows that $(I+S)\setminus I=\{k_im+i\}=\{a\}$. Thus, $I=(I+S)\setminus \{a\}$ and $I\cup\{a\}=I+S$. 

  As $I\cup\{a\}=(I+S)$, there cannot be another ideal $J$ with $I+S\prec J\prec I$, because that would lead to $I+S\subsetneq J\subsetneq I$, and $|(I+S)\setminus I|$ would be greater than two, a contradiction. This proves that $I+S$ covers $I$.
\end{proof}

\begin{figure}
  \raisebox{4.5em}{%
    \includegraphics[scale=0.25]{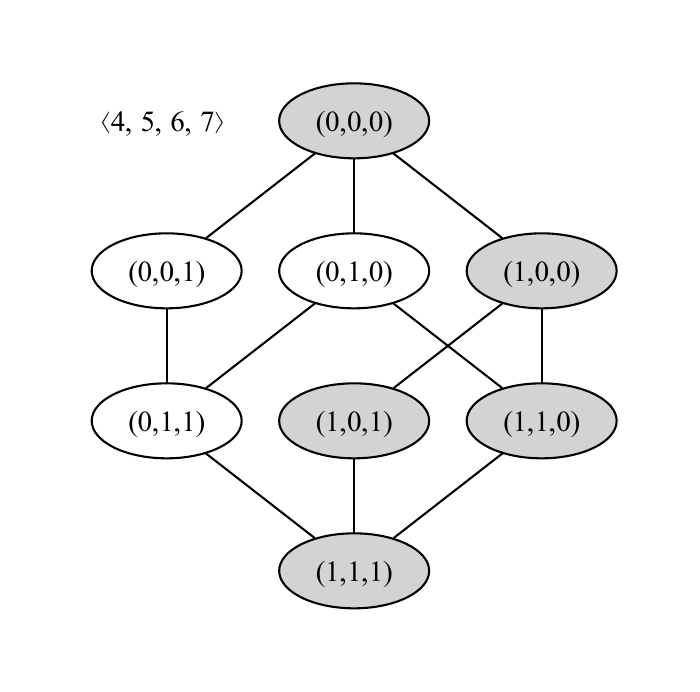} }
  \raisebox{3em}{%
    \includegraphics[scale=0.25]{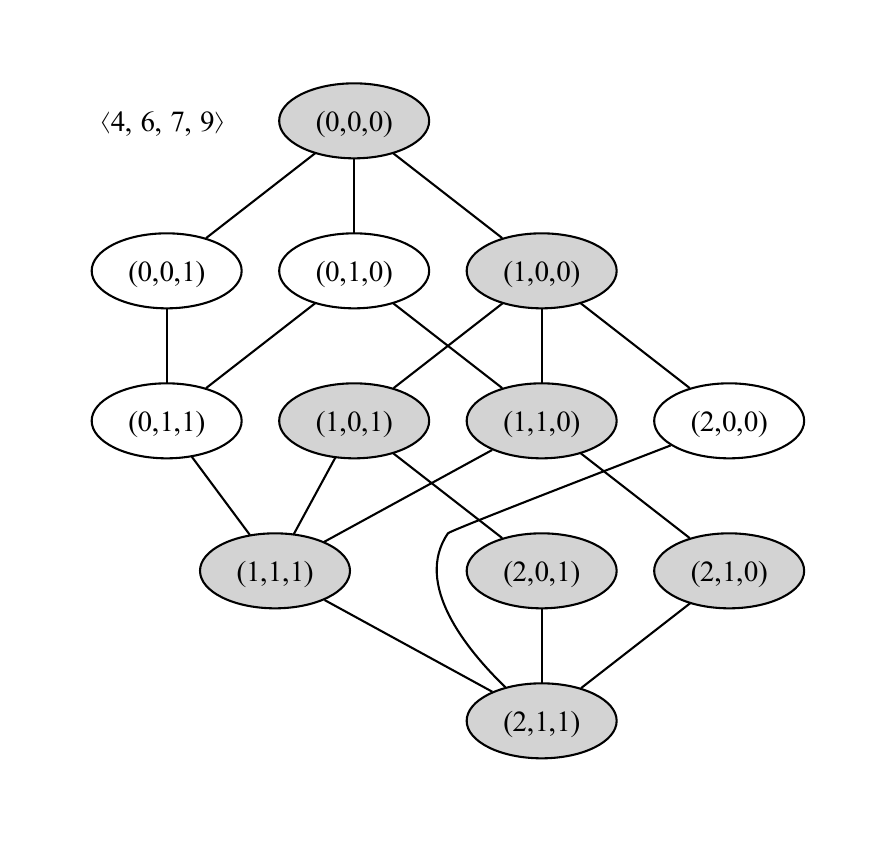} }
  \raisebox{1.5em}{%
    \includegraphics[scale=0.25]{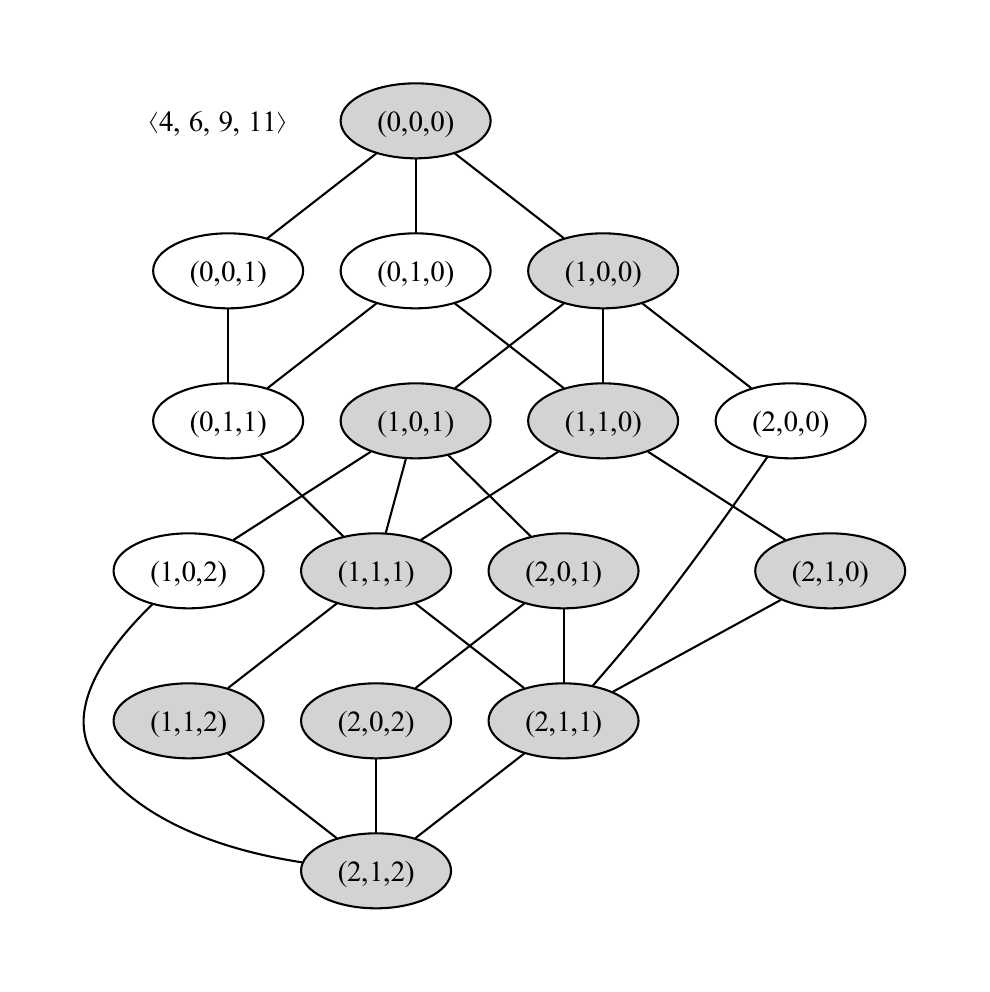} }
    \includegraphics[scale=0.25]{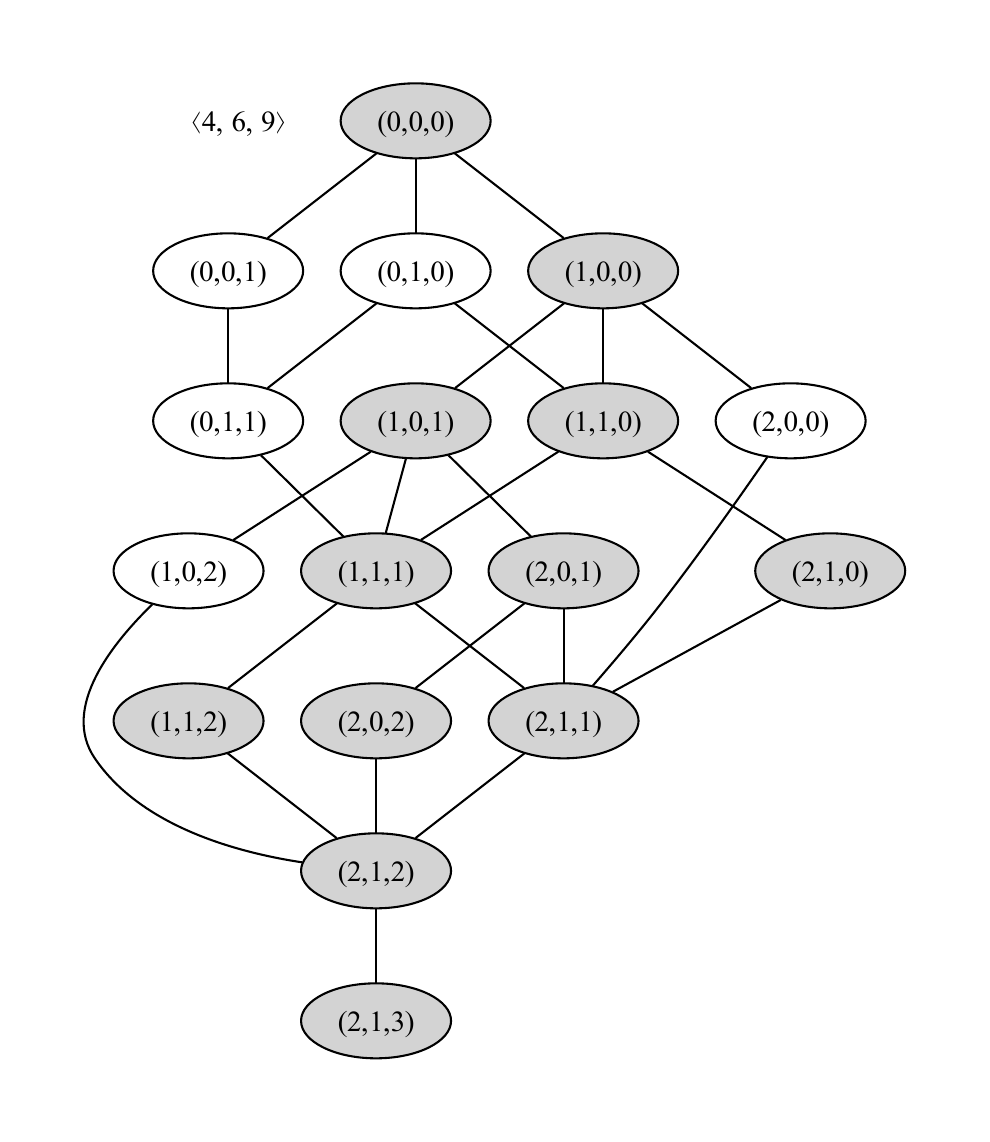}
    
\caption{Evolution of $\Ni$ after removing a minimal generator.}    
\label{fig:evolution}
\end{figure}

The following example shows that the condition $a>f$ in the previous result is necessary.

\begin{example}
    Let $S=\langle 3,19,23\rangle$, and let $a=19$. Then $i=1$. The set of  Kunz coordinates of the ideals in $\Ni\subset \Ni[S\setminus\{a\}]$ is 
    \[
    \{ ( 0, 6 ), ( 1, 7 ), ( 6, 0 ), ( 6, 1 ), ( 6, 2 ), ( 6, 3 ), ( 6, 4 ), ( 6, 5 ), ( 6, 6 ), ( 6, 7 )\}.
    \]
    The ideals not covered by Proposition~\ref{prop:normalized-ideals-unitary-extension} are $(0,6)_{\mathcal{K}}$ and $(1,7)_{\mathcal{K}}$.
\end{example}

\begin{remark}\label{rem:I+S-ideal-S}
    Let $S$ and $a$ be as in Proposition~\ref{prop:normalized-ideals-unitary-extension}. Then, for every $I\in \Ni[S\setminus\{a\}]$, we have $I+S\in \Ni$. If $I\in \Ni[S\setminus\{a\}]\setminus \Ni$, then by Proposition~\ref{prop:normalized-ideals-unitary-extension}, $I+S\in \Ni$, and if $I\in \Ni$, then $I+S=I\in \Ni$.
\end{remark}

\begin{remark}\label{rem: sobre ideales normalizados}
    Observe that in Proposition~\ref{prop:normalized-ideals-unitary-extension}, if $I$ is an ideal of $S\setminus\{a\}$ not containing $a$, then $I=(I+S)\setminus\{a\}$ and $I+S\in \Ni$. Also, if $J$ is another ideal of $S\setminus\{a\}$ and $I\preceq J$, then there exists $K\in \Ni[S\setminus\{a\}]$ such that $I+K=J$. Consequently, $I+S+K=I+S+K+S=J+S$. By Remark~\ref{rem:I+S-ideal-S}, $K+S\in \Ni$, and so $I+S\preceq J+S$ in $\Ni$. 
    If, in addition, $a\in J$, then $J$ is an ideal of $S$ and $I+S\preceq J$.
\end{remark}

Recall that a numerical semigroup $S$ is called \emph{ordinary} when its multiplicity $m$ coincides with the conductor of $S$, that is, $S = \mathbb{N}\setminus\{1,2,\dots, m-1\}$. In the remaining part of the manuscript, we will denote by $H_m$ the ordinary numerical semigroup with multiplicity $m$.

\begin{remark}\label{rem: numero de ideales de sem ordinario}
By \cite[Proposition 3.4]{apery-icm}, the set of ideals $\Ni$ of a numerical semigroup is in bijective correspondence with number of antichains in the poset $(\operatorname{G}(S), \leq_{S})$. For $H_m$, we have $\operatorname{G}(H_m)=\{1,2,\dots, m-1\}$ and therefore $|\mathfrak{I}_0(H_m)| = 2^{m-1}$. Moreover, if $I\in \IOm$ has Kunz coordinates $(x_1,\dots, x_{m-1})_{\mathcal{K}}$, then $x_i\in\{0,1\}$, for every $i\in\{1,\dots,m-1\}$, and $x_i = 0$ if and only if $i\in I$. In other words, the poset $(\IOm,\subseteq) $ is a Boolean algebra (for every $m\in\mathbb{N}$).
\end{remark}

\begin{remark}\label{rem: coordenadas de una suma-caso ordinario}
The equation \eqref{eq:kunz-coordinates-sum-general} expressing the Kunz coordinates of the sum of two ideals $I = (x_1,\dots, x_{m-1})_{\mathcal{K}}$, $J=(y_1 , \dots, y_{m-1})_{\mathcal{K}}$ can be significantly simplified when $S$ is ordinary. Indeed, if $I+J=(z_1,\dots,z_{m-1})_{\mathcal{K}}$, then for every $i\in \{1,\dots,m-1\}$,
\begin{equation}\label{eq: coordinadas de Kunz-suma de ideales-caso ordinario}
z_i=\min(\{x_{i_1}+y_{i_2} : i_1,i_2\in \{0,\dots,i\}, i_1+i_2 = i \}).    
\end{equation}
Notice that if $i_1,i_2\in \{0,\dots,m-1\}$ are such that $i_1+i_2=i+m$, then $\lfloor \frac{i_1+i_2}{m}\rfloor=1$, and consequently $x_i+y_0+\lfloor \frac{i+0}{m}\rfloor = x_i\le 1 = \lfloor \frac{i_1+i_2}{m}\rfloor \le x_{i_i}+y_{i_2}+ \lfloor \frac{i_1+i_2}{m}\rfloor$.
\end{remark}

The following result extends, in some sense, the content of Proposition~\ref{prop:normalized-ideals-unitary-extension} to the posets $\IOm$ and $\IOsm$. 


\begin{proposition}\label{prop: ordinary I0(S) y I0(Smeno)}
 For every $m\in\mathbb{N}$, the following facts hold.
 \begin{enumerate}
     \item $\IOm\subsetneq\IOsm$.
     \item For $I = (x_1,\dots,x_m)_{\mathcal{K}}\in \IOsm$, $I\in \IOm$ if and only if $x_m = 0$. In particular, $I = (x_1,\dots,x_m)_{\mathcal{K}}\in \IOsm\setminus \IOm$ if and only if $x_m = 1$. 
     \item For every $I\in \IOsm\setminus \IOm$, $I + \Hm \in \IOm$.
 \end{enumerate}
\end{proposition}
\begin{proof}
Let $I\in \IOm$. Then, $I\subseteq I+H_{m+1}\subseteq I+H_m=I$, which yields $I=I+H_{m+1}$ and consequently $I\in H_{m+1}$. Notice that $H_{m+1}\in \IOsm\setminus \IOm$. This proves (1).

Now, take $I=(x_1,\dots,x_m)_{\mathcal{K}}\in \IOsm$. We know by the paragraph after \eqref{eq:kunz-inequalities-general} that $m=0\times (m+1)+m\in I$ if and only if $x_m\le 0$, or equivalently, $x_m=0$. 

If $I\in \IOm$, then $m\in I$, which forces $x_m=0$. For the converse, suppose that $x_m=0$. Then, $m\in I$, and so $H_m\subseteq I$, and this yields $I+H_m=I+(H_{m+1}\cup\{m\})=I\cup (m+I)=I$. Hence, $I\in \IOsm$. This proves (2), since the second part of this assertion follows from the fact that $x_m\in \{0,1\}$.

Let $I\in \IOsm\setminus \IOm$ with $I = (x_1,\dots,x_{m-1},1)_{\mathcal{K}}$ (in virtue of (2)).
Clearly, $\Hm$ has Kunz coordinates $(1,\dots, 1,0)_{\mathcal{K}}$ as an element of $\IOsm$. In light of Remark~\ref{rem: coordenadas de una suma-caso ordinario}, $I+\Hm=(x_1,\dots,x_{m-1},1)_{\mathcal{K}}+ (1,\dots,1,0)_{\mathcal{K}}=(x_1,\dots,x_{m-1},0)_{\mathcal{K}}$.
Hence, $I + \Hm \in \IOm$ by (2).
\end{proof}

\begin{remark}
Let $m$ be a positive integer, and let  $I = (x_1,\dots, x_{m-1})_{\mathcal{K}}\in\Ni[\Hm]$, then using \eqref{eq:inclusion-kunz-coord} and Remark \ref{rem: numero de ideales de sem ordinario}, $I^{c} = (x'_{1},\dots,x'_{m-1})_{\mathcal{K}}$ is a cover of $I$ with respect to inclusion if and only if there exists a unique $i\in\{1,\dots,m-1\}$ such that $x_{i} = 1$ and $x'_{i} = 0$ and $x'_j = x_j$, for all $j\in\{1,\dots m-1\}\setminus\{i\}$. This can be seen as a rephrasing of \cite[Lemma~10]{isom-icm} for the particular case of ordinary semigroups with the use of \eqref{eq:genus-ideal-kunz}.  
\end{remark}

\section{When the poset of normalized ideals is a lattice}\label{sec:lattice}

In \cite[Theorem~20]{m-3-poni}, we proved that for $S$ a numerical semigroup with multiplicity three, the poset $(\Ni,\preceq)$ is a lattice, providing also an explicit description of its lattice operations. 

In this section, we characterize numerical semigroups $S$ fulfilling that $(\Ni,\preceq)$ is a lattice. We prove that $(\Ni,\preceq)$ is a lattice if and only if the multiplicity of $S$ does not exceed four.

From the content of \cite[Theorem 20]{m-3-poni}, one might reasonably think that $I\cup J$ or $I+J$ could act as $I\vee J$, but this is not the case in general. There are ideals $I$ and $J$ for which $I\vee J$ exists, however it is neither $I\cup J$ nor $I+J$, as shown by the following.

\begin{example}\label{ex:join-not-sum-nor-union}
  Let $S=\langle 4,9\rangle$. Then, $(\Ni,\preceq)$ is a join semilattice. We also have that 
  \begin{itemize}
    \item $(\{ 0, 1, 2 \}+S)\vee (\{ 0, 1, 6, 7 \}+S)= (\{ 0, 1, 2 \}+S)\cup (\{ 0, 1, 6, 7 \}+S)\neq (\{ 0, 1, 2 \}+S)+ (\{ 0, 1, 6, 7 \}+S)$,
    \item $(\{ 0, 1, 2 \}+S)\vee (\{ 0, 2, 5 \}+S)= (\{ 0, 1, 2 \}+S)+(\{ 0, 2, 5 \}+S)=\mathbb{N}\neq (\{ 0, 1, 2 \}+S)\cup (\{ 0, 2, 5 \}+S)$,
    \item $(\{ 0, 1, 2 \}+S)\vee (\{ 0, 1, 6\}+S)=\{0,1,2,7\}+S$, which is  not equal to $(\{ 0, 1, 2 \}+S)+ (\{ 0, 1, 6\}+S)$  nor to $(\{ 0, 1, 2 \}+S)\cup (\{ 0, 1, 6\}+S)$.
  \end{itemize}
\end{example}

The following example will be central to prove that for any numerical semigroup $S$ with multiplicity greater than four $(\Ni,\preceq)$ is not a lattice.

\begin{example}\label{ex:56789}
  The set of minimal upper bounds for $\{\{0,1\}+H_5,\{0,1,3\}+H_5\}$ is $\{\{ 0, 1, 2, 3 \}+H_5, \{ 0, 1, 3, 4 \}+H_5\}$, and so there is no supremum for the pair of elements $\{0,1\}+H_5$ and $\{0,1,3\}+H_5$ in $\Ni[H_5]$.
\end{example}

Let $S$ be a numerical semigroup, and let $T$ be an oversemigroup of $S$, that is, an idempotent element in $\Ni$ \cite[Proposition~5.14]{apery-icm}. Define
\[
  C_T=\{ I\in \Ni : I+T=I\} \text{  and  } \uparrow\! T = \{I \in \Ni : T\preceq I\}.
\]
Recall that $\uparrow\! T=C_T=\mathfrak{I}_0(T)$ \cite[Lemma 7]{m-3-poni}.

\begin{proposition}
  Let $S$ be a numerical semigroup with multiplicity greater than or equal to five. Then, $(\Ni,\preceq)$ is not a join semilattice nor a meet semilattice.
\end{proposition}
\begin{proof}
  Since $(\Ni,\preceq)$ is a finite poset, by \cite[Theorem 2.4]{Nation}, it suffices to show that $(\Ni,\preceq)$ is not a meet semilattice. 
  
  As $\operatorname{m}(S)\ge 5$, we have that $H_5$ is an oversemigroup of $S$. The poset $(\mathfrak{I}_0(H_5),\preceq)$ is not a join semilattice (see Example~\ref{ex:56789}) and, in particular, there is no supremum for the elements $\{0,1\}+H_5$ and $\{0,1,3\}+H_5$; since these ideals are in $\uparrow\! H_5=\Ni[H_5]$, its supremum in $\Ni$ should also be in $\uparrow\! H_5$. 
\end{proof}

A sort of converse to Remark \ref{rem: sobre ideales normalizados} is covered in the following technical lemma.

\begin{lemma}\label{lem:preceq-downwards}
    Let $S$ be a numerical semigroup with multiplicity $m$ and let $a$ be a minimal generator of $S$ with $m\neq a > \operatorname{F}(S) $. 
    Let $I$ and $J$ be two ideals of $S$ such that $a$ is a minimal generator of $I$ and of $J$ as ideals of $S\setminus\{a\}$. If $I+K=J$ for some ideal $K$ of $S$, then $a$ is a minimal generator of $K$ (as an ideal of $S\setminus\{a\}$) and $(I\setminus\{a\})+(K\setminus\{a\})=J\setminus\{a\}$.
\end{lemma}
\begin{proof}
Let $i=a\bmod m$. We can write $a=k_i m +i$, if $(k_1,\dots,k_{m-1})$ are the Kunz coordinates of $S$. 
    
If $a$ is not a minimal generator of $K$ as an ideal of $S\setminus\{a\}$, then $a=x+s$ for some $x\in K$ and $s\in S\setminus\{a\}$. But then $x\in J$ (as $K\preceq J$, we have $K\subseteq J$), which contradicts the fact that $a$ is a minimal generator of $J$.

As $a$ is a minimal generator of $I$, $J$ and $K$, we deduce that $I\setminus\{a\}, J\setminus\{a\},K\setminus\{a\}\in \Ni[S\setminus \{a\}]$.
Notice that $I\setminus\{a\}, J\setminus\{a\},K\setminus\{a\}\not\in \Ni$. For instance, $(I\setminus\{a\})+S\neq I\setminus\{a\}$ because $a=0 + a\in (I\setminus\{a\})+S$ and $a\not\in I\setminus\{a\}$.  

By Proposition~\ref{prop:normalized-ideals-unitary-extension}, if $I\setminus\{a\}=(x_1,\dots,x_{m-1})_{\mathcal{K}}$, $K\setminus\{a\}=(y_1,\dots,y_{m-1})_{\mathcal{K}}$, and $J\setminus\{a\}=(z_1,\dots,z_{m-1})_{\mathcal{K}}$, then $x_i=y_i=z_i=k_i+1$. If $I=(x_1',\dots,x_{m-1}')_{\mathcal{K}}$, $K=(y_1',\dots,y_{m-1}')_{\mathcal{K}}$ and $J=(z_1',\dots,z_{m-1}')_{\mathcal{K}}$ (seen as elements of $\Ni$), then $(x_1',\dots,x_{m-1}')=(x_1,\dots,x_{m-1})-\mathbf{e}_i$, $(y_1',\dots,y_{m-1}')=(y_1,\dots,y_{m-1})-\mathbf{e}_i$, and $(z_1',\dots,z_{m-1}')=(z_1,\dots,z_{m-1})-\mathbf{e}_i$, where $\mathbf{e}_i$ is the $i$th row of the $(m-1)\times(m-1)$ identity matrix. Also, $(I\setminus\{a\})+S=I$, $(J\setminus\{a\})+S=J$, and $(K\setminus\{a\})+S=K$. 

By hypothesis, 
\[
I+K=(x_1',\dots,x_{m-1}')_{\mathcal{K}}+(y_1',\dots,y_{m-1}')_{\mathcal{K}}=(z_1',\dots,z_{m-1}')_{\mathcal{K}}=J.
\]
By \eqref{eq:kunz-coordinates-sum-general},  $z_l'=\min(\{ x_{l_1}'+y_{l_2}'+\lfloor \frac{l_1+l_2}{m}\rfloor : l_1,l_2\in \{0,\dots,m-1\}, l_1+l_2\equiv l \pmod{m} \})$. If we want to prove that 
\[
(I\setminus\{a\})+(K\setminus\{a\})=(x_1,\dots,x_{m-1})_{\mathcal{K}}+(y_1,\dots,y_{m-1})_{\mathcal{K}}=(z_1,\dots,z_{m-1})_{\mathcal{K}}=J\setminus\{a\},
\]
then we have to show that for every $l\in\{1,\dots,m-1\}$,
\[z_l=\min\left(\left\{ x_{l_1}+y_{l_2}+\left\lfloor \frac{l_1+l_2}{m}\right\rfloor : l_1,l_2\in \{0,\dots,m-1\}, l_1+l_2\equiv l \tpmod{m} \right\}\right).\]
First, observe that since $a\not\in J\setminus\{a\}$, $a\neq u+v$ for any $u\in I\setminus\{a\}$ and $v\in K\setminus\{a\}$, and consequently $a\not\in \Ap(I\setminus\{a\})+\Ap(K\setminus\{a\})$. Hence, $x_{i_1}+y_{i_2}+\lfloor \frac{i_1+i_2}{m}\rfloor > k_i$ for all $i_1,i_2\in \{1,\dots,m-1\}$ such that $i_1+i_2\equiv i \tpmod{m}$. This, in particular, implies that $x_{i_1}+y_{i_2}+\lfloor \frac{i_1+i_2}{m}\rfloor \ge k_i+1$ for all $i_1,i_2\in \{1,\dots,m-1\}$ such that $i_1+i_2\equiv i \pmod{m}$. Therefore, $z_i=k_i+1=x_i+y_0=x_0+y_i=x_i=y_i= \min(\{ x_{i_1}+y_{i_2}+\lfloor \frac{i_1+i_2}{m}\rfloor : i_1,i_2\in \{0,\dots,m-1\}, i_1+i_2\equiv i \tpmod{m} \}) $.

Now, take $l\in\{1,\dots,m-1\}\setminus\{i\}$. Then, 
\[z_l'= \min(\{ x_{l_1}'+y_{l_2}'+\lfloor \frac{l_1+l_2}{m}\rfloor : l_1,l_2\in \{0,\dots,m-1\}, l_1+l_2\equiv l \pmod{m} \}).\] 
It follows that 
\begin{align}
    z_l'& = \min\left(\left\{ x_{l_1}'+y_{l_2}'+\left\lfloor \frac{l_1+l_2}{m}\right\rfloor : l_1,l_2\in \{0,\dots,m-1\}, l_1+l_2\equiv l \pmod{m} \right\}\right)\notag \\ 
    & =\min\left(\left\{ x_{l_1}'+y_{l_2}'+\left\lfloor \frac{l_1+l_2}{m}\right\rfloor : l_1,l_2\in \{0,\dots,m-1\}\setminus\{i\}, l_1+l_2\equiv l \tpmod{m} \right\}\right) \label{eq:zzp-1}\\ 
    & =\min\left(\left\{ x_{l_1}+y_{l_2}+\left\lfloor \frac{l_1+l_2}{m}\right\rfloor : l_1,l_2\in \{0,\dots,m-1\}\setminus\{i\}, l_1+l_2\equiv l \tpmod{m} \right\}\right) \notag \\ 
    & =\min\left(\left\{ x_{l_1}+y_{l_2}+\left\lfloor \frac{l_1+l_2}{m}\right\rfloor : l_1,l_2\in \{0,\dots,m-1\}, l_1+l_2\equiv l \tpmod{m} \right\}\right)=z_l. \label{eq:zzp-2}
\end{align} 
Equality \eqref{eq:zzp-1} follows by applying  \eqref{eq:kunz-inequalities-general} to $I$ and $K$:
\begin{itemize}
    \item if $l_1=i$, then $x_{l_1}'+y_{l_2}'=k_i+y_{l_2}'+ \lfloor \frac{l_1+l_2}{m}\rfloor\ge y_{(i+l_2)\bmod m}' = x_0'+y_l'$;
    \item if $l_2=i$, then $x_{l_1}'+y_{l_2}'=x_{l_1}'+k_i+ \lfloor \frac{l_1+l_2}{m}\rfloor \ge x_{(i+l_2)\bmod m}' = x_l'+y_0'$.
\end{itemize}
While \eqref{eq:zzp-2} holds by applying \eqref{eq:kunz-inequalities-general}  to $I\setminus\{a\}$ and $K\setminus\{a\}$:
\begin{itemize}
    \item if $l_1=i$, then $x_{l_1}+y_{l_2}=(k_i+1)+y_{l_2}+ \lfloor \frac{l_1+l_2}{m}\rfloor\ge y_{(i+l_2)\bmod m} = x_0+y_l$;
    \item if $l_2=i$, then $x_{l_1}+y_{l_2}=x_{l_1}+(k_i+1)+ \lfloor \frac{l_1+l_2}{m}\rfloor \ge x_{(i+l_2)\bmod m} = x_l+y_0$.
\end{itemize}
This proves that $I\setminus\{a\}+K\setminus\{a\}=J\setminus\{a\}$.
\end{proof}

Notice that the previous result is false if $a$ is chosen to be the multiplicity of $S$, as shown by the following example, inspired by the Hasse diagram of $H_4$
drawn in Figure~\ref{fig:evolution}. 

\begin{example}
    Let $S=\langle 3,4,5\rangle$ and let $a=3=\operatorname{m}(S)$. Set $T=S\setminus\{3\}=\langle 4,5,6,7\rangle$. 
    
    Let $I=\{0,2\}+S=\{0,2,3\}+T$ and $J=\{0,1,2\}+S=\{0,1,2,3\}+T=\mathbb{N}$. Then, $3$ is a minimal generator of $I$ and $J$ as ideals of $T$.  Take $K=\{0,1\}+S=\{0,1,3\}+T$. Then,  $I+K=J$, and so $I\preceq J$. However, $I\setminus\{3\}=\{0,2\}+T$, $J\setminus\{3\}=\{0,1,2\}+T$, $K\setminus\{3\}=\{0,1\}+T$, and $(I\setminus\{3\})+(K\setminus\{3\})\neq (J\setminus\{3\})$. As a matter of fact, $I\setminus\{3\} \not\preceq J\setminus\{3\}$.

    Notice that $I=(1,0,0)_{\mathcal{K}}=(1,0)_{\mathcal{K}}$, $I\setminus\{3\}=(1,0,1)_{\mathcal{K}}$,  $J=(0,0,0)_{\mathcal{K}}=(0,0)_{\mathcal{K}}$, and $J\setminus\{3\}=(0,0,1)_{\mathcal{K}}$.
\end{example}

We now have all the ingredients to prove that the lattice property on $\Ni$ is preserved once we remove a minimal generator of $S$ different from its multiplicity.

\begin{theorem}\label{thm:lattice-removing-minimal-generator}
Let $S$ be a numerical semigroup such that $(\Ni,\preceq)$ is a lattice, and let $a$ be a minimal generator of $S$, with $m\neq a > \operatorname{F}(S) $. Then, $(\Ni[S\setminus \{a\}],\preceq)$ is a lattice as well.
\end{theorem}
\begin{proof}
Let $I,J\in \Ni[S\setminus\{a\}]$. We distinguish several cases, depending on whether $a$ belongs or not to $I$ and $J$.

If $a\in I\cap J$, then by Proposition~\ref{prop:normalized-ideals-unitary-extension}, $I$ and $J$ are in $\Ni$, and so $I\vee J$ exists, because all upper bounds of $I$ and $J$ are in $\Ni$, and $(\Ni,\preceq)$ is a lattice by hypothesis.

Suppose that $a\in I$ and $a\not\in J$. Let $K\in \Ni[S\setminus\{a\}]$ be such that $I\preceq K$ and $J\preceq K$. Hence, $I\subseteq K$ and consequently $a\in K$. By Proposition~\ref{prop:normalized-ideals-unitary-extension}, $I,K\in \Ni$ and $J+S\preceq K+S=K$. As $J+S\in \Ni$ (by Proposition~\ref{prop:normalized-ideals-unitary-extension}), $I\vee (J+S)\preceq K$. This proves that $I\vee J=I\vee (J+S)$.

The case $a\not\in I$ and $a\in J$ is similar: $I\vee J=(I+S)\vee J$.

Finally, suppose that $a\not\in I\cup J$. If every $K\in \Ni[S\setminus\{a\}]$ with $I\preceq K$ and $J\preceq K$ is in $\Ni$, then $I+S\preceq K+S=K$ and $J+S\preceq K+S=K$, and so $I\vee J=(I+S)\vee (J+S)$.

Thus, the remaining case is when there is $K\in\Ni[S\setminus\{a\}]\setminus \Ni$ ($a\not \in K$ in virtue of Proposition~\ref{prop:normalized-ideals-unitary-extension}), such that $I\preceq K$ and $J\preceq K$. Notice that in this setting $a$ is a minimal generator of $K+S=K\cup\{a\}$ as an ideal of $S\setminus\{a\}$. 

As $(\Ni\preceq)$ is a lattice and both $I+S$ and $J+S$ are ideals of $S$, there exists $M=(I+S)\vee (J+S)$, which is an ideal of $S$. Since $I\preceq K$ and $J\preceq K$, we deduce that $I+S\preceq K+S$ and $J+S\preceq K+S$; therefore $M\preceq K+S=K\cup \{a\}$. 

Notice that $a\in M$ ($M$ contains $I+S=I\cup\{a\}$). Suppose that $a$ is not a minimal generator of $M$ (as an ideal of $S\setminus\{a\})$. Then, there exists $g\in M$ and $s\in S\setminus\{a\}$ with $s\neq 0$ and $a=g+s$. But then $g\in K$, and so $a=g+s\in K$, a contradiction.  Therefore, $a$ is a minimal generator of $M$, and consequently the set  $M\setminus\{a\}$ is an ideal of $S\setminus\{a\}$. Now, we can apply Lemma~\ref{lem:preceq-downwards} to $I+S\preceq M$ and $J+S\preceq M$, to derive that $I\preceq M\setminus\{a\}$ and $J\preceq M\setminus\{a\}$; also, from $M\preceq K+S$ to deduce that $M\setminus\{a\}\preceq K$. This proves that $M\setminus\{a\}\preceq K$, for every upper bound of $I$ and $J$ (with respect to $\preceq$) not containing $a$. If $K$ is an upper bound of $I$ and $J$ containing $a$, then $I+S\preceq K+S=K$ and $J+S\preceq K+S=K$, and so $M\preceq K$. Since $(M\setminus\{a\})+S=M$, we deduce that $M\setminus\{a\}\preceq M\preceq K$. This proves that $M\setminus\{a\}=I\vee J$.
\end{proof}

\begin{corollary}\label{cor:lattice-mult-four}
    Let $S$ be a numerical semigroup with multiplicity four. Then, $(\Ni,\preceq)$ is a lattice.
\end{corollary}
\begin{proof}
    Let $H_4$ be the ordinary numerical semigroup of multiplicity 4 (we are using the same notation introduced in Section \ref{sec: normalized ideals}). 
The Hasse diagram of $(\Ni[H_4],\preceq)$ is shown in Figure~\ref{fig:evolution} (first diagram on the left; Kunz coordinates are used to display the vertices of the diagram). Clearly, $(\Ni[H_4],\preceq)$ is a lattice. If $S=H_4$, then we are done. Otherwise, $S$ is strictly contained in $H_4$. Let $S=S_1\subset \dots \subset S_k=H_4$ be the sequence of numerical semigroups obtained by applying \cite[Lemma~6]{ns-app}. Then, for every $i\in \{1,\dots,k-1\}$, $S_{i+1}=S_i\cup\{g_i\}$, with $g_i=\max(H_4\setminus S_i)=\operatorname{F}(S_i)$, and $S_i=S_{i+1}\setminus\{g_i\}$. Observe that $g_i$ is a minimal generator of $S_{i+1}$ greater than four; also $g_i=\max(H_4\setminus S_i)> \max(H_4\setminus S_{i+1})=\operatorname{F}(S_{i+1})=g_{i+1}$. Thus, we can apply Theorem~\ref{thm:lattice-removing-minimal-generator} to obtain that $(\Ni[S_i],\preceq)$ is a lattice for all $i$.
\end{proof}

If $S$ has multiplicity two, then \cite[Example~5.9]{apery-icm} ensures that $(\Ni,\preceq)$ is a totally ordered set, and thus a lattice. Also, recall that \cite[Theorem~20]{m-3-poni} states that if $S$ is a numerical semigroup with multiplicity three, then $(\Ni,\preceq)$ is a lattice.  Putting all this together, we obtain the following consequence.

\begin{corollary}\label{cor: reticulo iff m menorigual 4}
    Let $S$ be a numerical semigroup. Then, $(\Ni,\preceq)$ is a lattice if and only if $\operatorname{m}(S)\le 4$.
\end{corollary}

The following shows that, in general, the lattice $(\Ni, \preceq)$ is not distributive (for both multiplicity three and four).

\begin{remark}
It can be shown that the set $\{ H_4, \{0,3\}+H_4,\{0,2,3\}+H_4,\{0,1,2\}+H_4,\mathbb{N}\}$ is a sublattice of $\Ni[H_4]$ isomorphic to a a pentagon, and thus $\Ni[H_4]$ cannot be a distributive lattice. Every numerical semigroup $T$ with multiplicity larger than three will have $H_4$ as an oversemigroup. Thus, if $\Ni[T]$ is a lattice, it will contain the sublattice $\Ni[T]$, and thus $\Ni[T]$ will not be a distributive lattice. 
\end{remark}

\begin{figure}
    \centering
    \includegraphics[scale=0.5]{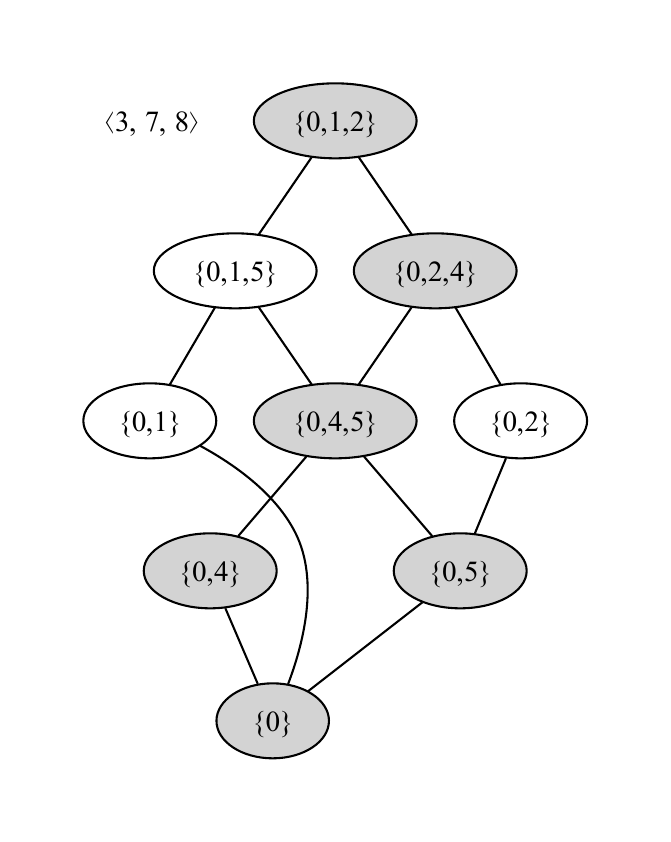}
    \caption{The non-distributive lattice $(\mathcal{I}_0(T),\preceq)$, for $T =\langle 3,7,8\rangle$.}
    \label{fig:3_7_8}
\end{figure}

\begin{remark}
    Let $S$ be a numerical semigroup with multiplicity three in which $\preceq$ does not equal $\subseteq$. Then, by \cite[Proposition 12]{m-3-poni}, $\{4,5\}\cap S=\emptyset$. Thus, $T=\langle 3,7,8\rangle= \{0,3\}\cup (6+\mathbb{N})$ is an oversemigroup of $S$. Let $D=\{T,\{0,4\}+T,\{0,4,5\}+T,\{0,1,5\}+T,\{0,1\}+T\}$ (see Figure~\ref{fig:3_7_8}). Then, $D$ is a sublattice of $\Ni[T]$ isomorphic to a pentagon, and consequently neither $\Ni[T]$ nor $\Ni$ are distributive.
\end{remark}

\section{Irreducibility}\label{sec:irred}

Let $S$ be a numerical semigroup. 
As $\Ni$ has finitely many elements, if  $(\Ni,\preceq)$ is lattice (by Corollary \ref{cor: reticulo iff m menorigual 4} this is the case if and only if $m(S)\leq 4$), then any of its elements is a join of $\vee$-irreducible elements, and also a meet of $\wedge$-irreducible elements.

Observe that $(\Ni,\subseteq)$ is a lattice where meet equals intersection of ideals and join is the union of ideals. Thus, it also makes sense to think about $\cap$- and $\cup$-irreducible ideals.

An ideal $I\in \Ni\setminus \{S\}$ is $+$-\emph{irreducible} if $I = J + K$ implies $I = J$ or $I = K$. In \cite{apery-icm} these ideals are called irreducible. We use the $+$ prefix to distinguish them from $\vee$-irreducibles, $\cap$-irreducibles, $\cup$-irreducibles and $\wedge$-irreducibles.

Define $\mathfrak{P}_{0} (S) = \{\{0,g\} + S :  g\text{ is a gap of } S\}$. In light of \cite[Proposition~3.1]{b-k}, we know that if $I$ is in $\mathfrak{P}_0(S)$, then $I$ is $+$-irreducible. The converse does not hold (see for instance \cite[Example~3.7]{b-k} or \cite[Example~5.3]{apery-icm}). 

\begin{remark}\label{rem:union-irreducible-principal}
    Let $S$ be a numerical semigroup. By Proposition~\ref{prop:irreducible-unique-cover}, $I$ is $\cup$-irreducible if and only if $I$ is the cover of at most one ideal in $\Ni$. By \cite[Lemma~11]{isom-icm}, the number of ideals covered by $I$ equals the number of non-zero minimal generators of $I$. Thus, $I$ is $\cup$-irreducible if and only if $I\in \mathfrak{P}_0(S)\cup\{S\}$.
\end{remark}

Next, we see that, in our context, $+$-irreducible implies $\vee$-irreducible.

\begin{lemma}\label{lem:sum-irr-join-irred}
    Let $S$ be a numerical semigroup such that $m(S)\leq 4$ and let $I\in\Ni$. If $I$ is $+$-irreducible, then it is $\vee$-irreducible.
\end{lemma}
\begin{proof}
Suppose that $I\in\Ni$ is $+$-irreducible and, to the contrary, that $I=J\vee K$ with $J,K\in \Ni\setminus\{I\}$ (observe that $J\vee K$ exists as $(\Ni,\preceq)$ is a lattice by Corollary \ref{cor: reticulo iff m menorigual 4}). 
As $J\preceq J\vee K=I$ and $K\preceq J\vee K=I$, there exists $N,M\in \Ni$ such that $I=J+N=K+M$. By taking into account that $I$ is $+$-irreducible and $J\neq I\neq K$, we deduce that $I=N=M$, and consequently $I=J+I=K+I$. Thus, $J+K+I=J+I=I$, and so $J+K\preceq I$. Also, $I=J\vee K\preceq J+K$, since $J+K$ is an upper bound for $\{J,K\}$ with respect to $\preceq$, and thus $J+K\preceq I\preceq J+K$, yielding $I=J+K$, contradicting that $I$ is $+$-irreducible. 
\end{proof}

Hence, for every $I\in \Ni$, $I\in \mathfrak{P}_0(S)$ implies $I$ is $+$-irreducible, which in turn forces $I$ to be $\vee$-irreducible (for $m(S)\leq 4$). As we mentioned above, the converse of the first implication does not hold. The following example illustrates that $\vee$-irreducibility does not imply $+$-irreducibility, and thus none of these implications can be reversed in general.

\begin{example}
    There are ideals that are $\vee$-irreducible but are not $+$-irreducible. Consider, for instance, $S=
    \langle 4,7,9\rangle$, and let $I=\{0,1,2\}+S$. Then, $I=(\{0,1\}+S)+(\{0,1\}+S)$ is $\vee$-irreducible. 
\end{example}

For multiplicity three, the above implications can be reversed.

\begin{proposition}\label{prop: vee-irred, irred y principales coinciden}
Let $S$ be a numerical semigroup with multiplicity three and $I\in\Ni\setminus S $. Then, the following are equivalent: 
\begin{enumerate}
    \item $I$ is $+$-irreducible.
    \item $I$ is $\vee$-irreducible;
    \item $I$ is $\cup$-irreducible;
    \item $I\in \mathfrak{P}_{0} (S)$.
    \end{enumerate}
\end{proposition}
\begin{proof}
We already know that the first assertion implies the second (Lemma~\ref{lem:sum-irr-join-irred}), the fourth implies the first (by \cite[Proposition~3.1]{b-k}), and the third is equivalent to the fourth (see Remark~\ref{rem:union-irreducible-principal}). Thus, it suffices to prove that the second statement implies the fourth. 

Suppose that $I\not\in\mathfrak{P}_0(S)$. Since $S$ has multiplicity three, $I = \{0, g_1 , g_2\} + S$, for some $g_1$, $g_2$ gaps of $S$, such that neither $g_2-g_1\in S$ nor $g_1-g_2\in S$ (otherwise, $I=\{0,g_1\}+S$ or $I=\{0,g_2\}+S$, leading to $I\in \mathfrak{P}_0(S)$). Consider $J = \{0, g_1\} + S$ and $K = \{0, g_2\} + S$. By \cite[Lemma 8]{isom-icm} and using the fact that $\{g_2-g_1,g_2-g_1\}\cap S=\emptyset$,  we deduce that $J\nsubseteq K$ and $K\nsubseteq J$. Therefore, $J$ and $K$ are incomparable elements in the lattice $(\Ni,\preceq)$ and so, by \cite[Theorem~20]{m-3-poni}, $J\vee K = J+ K = \{0, g_1, g_2, g_1 + g_2\} + S$. Notice that $\{0\bmod 3,g_1\bmod 3,g_2\bmod 3\}=\{0,1,2\}$, and so $g_1+g_2$ is either congruent with $0$, $g_1$ or $g_2$ modulo three. This implies that $\{0,g_1,g_2,g_1+g_2\}+S=\{0,g_1,g_2\}+S=I$, showing that $I=J+K$ and thus $I$ is not $\vee$-irreducible. 
\end{proof}

From the proof of this last result we derive the following consequence.

\begin{corollary}
Let $S$ be a numerical semigroup with multiplicity three and let $I\in \Ni$. Then, $I$ is either $+$-irreducible or there exist two $+$-irreducible ideals $J$ and $K$ such that $I = J\vee K$.
\end{corollary}

Let $S$ be a numerical semigroup and let $I\in \Ni$. The \emph{Frobenius number} of $I$ can be defined, exactly as for numerical semigroups, as $\operatorname{F}(I)=\max( \mathbb{Z}\setminus I)$. 

\begin{lemma}\label{lem:IcupFI}
    Let $S$ be a numerical semigroup and let $I\in \Ni\setminus\mathbb{N}$. Then, $\bar{I}=I\cup \{\operatorname{F}(I)\}\in \Ni$.
\end{lemma}
\begin{proof}
    We start by proving that $\bar{I}\in \Ni$, that is, $\bar{I}+S=\bar{I}$ and $0\in \bar{I}$. The second condition follows easily since $0\in I\subset \bar{I}$. As $I+S=I$, in order to prove that $\bar{I}+S=\bar{I}$, it suffices to show that $\operatorname{F}(I)+S\subseteq \bar{I}$. Clearly, $\operatorname{F}(I)+0=\operatorname{F}(I)\in \bar{I}$, and for every $s\in S\setminus \{0\}$, $\operatorname{F}(I)+s>\operatorname{F}(I)$, and thus $\operatorname{F}(I)+s\in I\subseteq \bar{I}$. 
\end{proof}

\begin{lemma}\label{lem:unique-cover-inclusion}
    Let $S$ be a numerical semigroup and let $I\in \Ni$ with a unique cover with respect to inclusion. For every $J\in \Ni\setminus\{I\}$ , if  $I\preceq J$, then $\bar{I}\preceq J$, where $\bar{I}=I\cup \{\operatorname{F}(I)\}$.
\end{lemma}
\begin{proof}    
    Notice that as $I$ has a unique cover (with respect to inclusion) and $\bar{I}=I\cup\{\operatorname{F}(I)\}$ covers $I$, we deduce that  $\bar{I}$ is precisely the unique cover of $I$.
    
    Suppose that $I\prec J$ for some $J\in \Ni$.  As $I\subseteq J$, and $I\neq J$, $\bar{I}\subseteq J$, and in particular $\operatorname{F}(I)\in J$.
    From $I\preceq J$, we deduce that there exists $K\in \Ni$ such that $I+K=J$. Let us show that $\bar{I}+K=J$. Clearly, $J=I+K\subseteq \bar{I}+K$. For the other inclusion, take $k\in K$ and $i\in \bar{I}$. If $i\in I$, then $i+k\in I+K=J$. If $i=\operatorname{F}(I)$ and $k=0$, then $i+k=\operatorname{F}(I)$, which we already know belongs to $J$. Finally, for $i=\operatorname{F}(I)$ and $k\in K\setminus\{0\}$, we get that $i+k>\operatorname{F}(I)$, and so $i+k\in I\subseteq J$.
\end{proof}

\begin{proposition}
    Let $S$ be a numerical semigroup such that $(\Ni,\preceq)$ is a lattice, and let $I\in \Ni$. If $I$ is $\cap$-irreducible, then $I$ is $\wedge$-irreducible.
\end{proposition}
\begin{proof}
    As $I$ is $\cap$-irreducible, we know by Proposition~\ref{prop:irreducible-unique-cover} that $I$ has a unique cover, which by Lemma~\ref{lem:IcupFI} must be $\bar{I}=I\cup\{\operatorname{F}(I)\}$.
Suppose that $I=J\wedge K$ with $J,K\in \Ni$. Then, $I\preceq J$ and $I\preceq K$, and by Lemma~\ref{lem:unique-cover-inclusion} we deduce that $\bar{I}\preceq J$ and $\bar{I}\preceq K$. But then $\bar{I}\preceq J\wedge K=I$, yielding $\bar{I}\subseteq I$, a contradiction.
\end{proof}

The other implication is not true as the following example shows.

\begin{example}
    Let $S=\langle 3,7\rangle$. Then, $I=\{0,4,8\}+S$ is $\wedge$-irreducible and it is not $\cap$-irreducible.
\end{example}

\section*{Acknowledgments}
The research was carried out thanks to the hospitality offered by the Institute of Mathematics (IMAG) of the University of Granada and the staff exchange program ERASMUS+ funded by the European Community.

S. Bonzio acknowledges the support by the Italian Ministry of Education, University and Research through the PRIN 2022 project DeKLA (``Developing Kleene Logics and their Applications'', project code: 2022SM4XC8) and the PRIN Pnrr project ``Quantum Models for Logic, Computation and Natural Processes (Qm4Np)'' (cod. P2022A52CR). He also acknowledges the Fondazione di Sardegna for the support received by the project MAPS (grant number F73C23001550007), the University of Cagliari for the support by the StartUp project ``GraphNet''. Finally, he gratefully acknowledges also the support of the INDAM GNSAGA (Gruppo Nazionale per le Strutture Algebriche, Geometriche e loro Applicazioni). \\
\noindent
P. A. García-Sánchez is partially supported by the grant number ProyExcel\_00868 (Proyecto de Excelencia de la Junta de Andalucía) and by the Junta de Andaluc\'ia Grant Number FQM--343. He also acknowledges financial support from the grant PID2022-138906NB-C21 funded by MICIU/AEI/10.13039/\bignumber{501100011033} and by ERDF ``A way of making Europe'', and from the Spanish Ministry of Science and Innovation (MICINN), through the ``Severo Ochoa and María de Maeztu Programme for Centres and Unities of Excellence'' (CEX2020-001105-M).

\end{document}